\documentclass[12pt,a4paper]{amsart}

\usepackage{amsmath}
\usepackage{amssymb}
\usepackage{amsfonts}
\usepackage{amsthm}
\usepackage[margin=4cm]{geometry}
\usepackage{setspace}
\usepackage{hyperref}
\usepackage{fixmath}
\usepackage{aliascnt}
\usepackage{comment}
\usepackage[usenames,dvipsnames]{color}
\usepackage{stmaryrd}
\usepackage{a4wide}
\usepackage[all]{xy}

\newtheoremstyle{property} %name
{3pt} %Space above
{3pt} %Space below
{\itshape} %Body font
{} %Indent amount
{\bfseries} %Theorem head font
{} %Punctuation after theorem head
{\newline} %Space after theorem head
{\thmname{#1} \thmnumber{#2} (\thmnote{#3})} %Theorem head spec (can be left empty, meaning `normal')

\theoremstyle{plain}
\newtheorem{theorem}{Theorem}[section] % reference by thm:
\theoremstyle{definition}
\theoremstyle{property}
\theoremstyle{remark}

\newaliascnt{proposition}{theorem}
\newaliascnt{lemma}{theorem}
\newaliascnt{corollary}{theorem}
\newaliascnt{observation}{theorem}
\newaliascnt{definition}{theorem}
\newaliascnt{fact}{theorem}
\newaliascnt{remark}{theorem}
\newaliascnt{example}{theorem}
\newaliascnt{property}{theorem}

\theoremstyle{plain}
\newtheorem{proposition}[proposition]{Proposition} % reference by prp:
\newtheorem{lemma}[lemma]{Lemma} % reference by lem:
\newtheorem{corollary}[corollary]{Corollary} % reference by cor:
\theoremstyle{definition}
\newtheorem{definition}[definition]{Definition} % reference by def:
\newtheorem{remark}[remark]{Remark}% reference by rmk:
\newtheorem*{definition*}{Definition} % do not reference unlabelled items
\newtheorem*{remark*}{Remark} % do not reference unlabelled items
\newtheorem*{example*}{Example} % do not reference unlabelled items
\theoremstyle{property}
\theoremstyle{remark}

\aliascntresetthe{proposition}
\aliascntresetthe{lemma}
\aliascntresetthe{corollary}
\aliascntresetthe{observation}
\aliascntresetthe{definition}
\aliascntresetthe{fact}
\aliascntresetthe{remark}
\aliascntresetthe{example}
\aliascntresetthe{property}

\providecommand{\bfbeta}{\mathbold{\beta}}

\title[Existential theory of henselian fields]{The existential theory of\\ equicharacteristic henselian valued fields}
\author{Sylvy Anscombe and Arno Fehm}
\thanks{\today\\During this research the first author was funded by EPSRC grant EP/K020692/1.}

\address{School of Mathematics, University of Leeds, Leeds LS2 9JT, United Kingdom}
\email{pmtwga@leeds.ac.uk}
\address{Fachbereich Mathematik und Statistik, University of 
Konstanz, 78457 Konstanz, Germany}
\email{arno.fehm@uni-konstanz.de}

\begin{document}

\begin{abstract}
We study the existential (and parts of the universal-existential) theory of equicharacteristic henselian valued fields.
We prove, among other things, an existential Ax-Kochen-Ershov principle,
which roughly says that the existential theory of an equi\-characteristic henselian valued field (of arbitrary characteristic) is determined
by the existential theory of the residue field; in particular, it is independent of the value group.
As an immediate corollary, we get an unconditional proof of the decidability of the existential theory of $\mathbb{F}_{q}((t))$.
\end{abstract}

\maketitle

\section{Introduction}

\noindent
We study the first order theory of a henselian valued field $(K,v)$ in the language of valued fields.
For {\bf residue characteristic zero}, this theory is well-understood through the celebrated {\em Ax-Kochen-Ershov (AKE) principles},
which state that, in this case, the theory of $(K,v)$ is completely determined by the theory of the residue field $Kv$
and the theory of the value group $vK$ (see e.g.~\cite[\S4.6]{PrestelDelzell}). 
In other words, if a sentence holds in one such valued field, then it 
holds in any other with elementarily equivalent residue field and value group (the {\em transfer principle}).
As a consequence, one gets that the theory of $(K,v)$ is {\em decidable} if and only if the theory of the residue field and the theory of the value group are decidable.

Some of this theory can be carried over to certain {\bf mixed characteristic} henselian valued fields
such as the fields of $p$-adic numbers $\mathbb{Q}_p$, whose theory was axiomatized and proven to be decidable by Ax-Kochen and Ershov in 1965.
However, for henselian valued fields of {\bf positive characteristic}, no such general principles are available. 
For example, in \cite{Kuhlmann01}, it is shown that the theory of characteristic $p>0$ henselian valued fields with value group elementarily equivalent to $\mathbb{Z}$ and residue field $\mathbb{F}_{p}$ is incomplete.
It is not known whether there is a suitable modification of the AKE principles that hold for arbitrary henselian valued fields of positive characteristic, and the decidability of the field of formal power series $\mathbb{F}_q((t))$ is a long-standing open problem.

For the first problem, the most useful approximations are AKE principles for certain classes of valued fields,
most notably F.-V.~Kuhlmann's recently published work \cite{Kuhlmann??} on the model theory of {\em tame fields}.
For the second problem, the best known result is by Denef and Schoutens from 2003, who proved in \cite{Denef-Schoutens03}
that resolution of singularities in positive characteristic would imply
that the {\em existential} theory of $\mathbb{F}_q((t))$ is decidable (i.e.~Hilbert's tenth problem for $\mathbb{F}_q((t))$ has a positive solution).

In this work, we take a different approach at deepening our understanding of the positive characteristic case:
Instead of limiting ourselves to certain classes of valued fields,
we attempt to prove results for arbitrary equicharacteristic henselian valued fields, 
but (having results like Denef-Schoutens in mind) instead restrict to existential or slightly more general sentences:
The technical heart of this work is a study of transfer principles for certain universal-existential sentences,
which builds on the aforementioned \cite{Kuhlmann??},
see the results in Section \ref{sec:transfer}.
While some of these general results will have applications for example in the theory of definable valuations
(see \cite{Anscombe-Koenigsmann14}, \cite{Cluckers-Derakhshan-Leenknegt-Macintyre12}, \cite{Fehm14}, \cite{Prestel14}
for some of the recent developments),
in this work we then restrict this machinery to existential sentences and deduce the following result (cf.~\autoref{thm:E.complete}):

\begin{theorem}\label{thm:intro.E.complete}
For any field $F$,
the theory $T$ of equicharacteristic henselian nontrivially valued fields with residue fields which model both the existential and universal theories of $F$ is
$\exists$-complete, i.e. for any existential sentence $\phi$ either $T\models\phi$ or $T\models\neg\phi$.
\end{theorem}

Note that the value group plays no role here: The existential theory of an equicharacteristic henselian nontrivially valued field
is determined solely by its residue field.
From this theorem, we obtain an AKE principle for $\exists$-sentences (cf.~\autoref{cor:existential.AKE}):

\begin{corollary}\label{cor:intro.existential.decidability}
Let $(K,v),(L,w)$ be equicharacteristic henselian nontrivially valued fields. 
If the residue fields $Kv$ and $Lw$ have the same existential theory,
then so do the valued fields $(K,v)$ and $(L,w)$.
\end{corollary}

Moreover, we conclude the following corollary on decidability (cf.~\autoref{cor:existential.decidability}):

\begin{corollary}
Let $(K,v)$ be an equicharacteristic henselian nontrivially valued field. The following are equivalent:
\begin{enumerate}
\item The existential theory of $Kv$ in the language of rings is decidable.
\item The existential theory of $(K,v)$ in the language of valued fields is decidable.
\end{enumerate}
\end{corollary}

\noindent 
As an immediate consequence, we get the first unconditional proof of the decidability
of the existential theory of $\mathbb{F}_{q}((t))$ (cf.~\autoref{cor:Fq((t))}).
Note, however, that the conditional result in \cite{Denef-Schoutens03} is for a language with a constant for $t$ --
Section \ref{sec:decidability} also contains a brief discussion of this difference. 

\section{Valued fields}

\noindent
For a valued field $(K,v)$ we denote by $vK=v(K^\times)$ its value group, by $\mathcal{O}_v$ its valuation ring
and by $Kv=\{av:a\in\mathcal{O}_v\}$ its residue field.
For standard definitions and facts about henselian valued fields we refer the reader to \cite{Engler-Prestel05}.
As a rule, if $L/K$ is a field extension to which the valuation $v$ can be extended uniquely, we denote also
this unique extension by $v$.
This applies in particular if $v$ is henselian, and
for the perfect hull $L=K^{\rm perf}$ of $K$.
We will make use of the following well-known fact:

\begin{lemma}\label{fact:complete.the.square}
Let $(K,v)$ be a valued field and let $F/Kv$ be any field extension. Then there is an extension of valued fields $(L,w)/(K,v)$ such that $Lw/Kv$ is isomorphic to the extension $F/Kv$.
\end{lemma}

\begin{proof}
See e.g. \cite[Theorem 2.14]{Kuhlmann04a}.
\end{proof}

Also the following lemma is probably well known, but for lack of reference we sketch a proof,
which closely follows \cite[Lemma 9.30]{Kuhlmann??book}:

\begin{definition}\label{def:section}
Let $(K,v)$ be a valued field. 
A \em partial section \rm (of the residue homomorphism) is a map $f:E\longrightarrow K$, for some subfield $E\subseteq Kv$, 
which is an $\mathcal{L}_{\mathrm{ring}}$-embedding such that $(f(a))v=a$ for all $a\in E$.
It is a {\em section} if $E=Kv$.
\end{definition}

\begin{lemma}\label{lem:section}
Let $(K,v)$ be an equicharacteristic henselian valued field, let $E\subseteq Kv$ be a subfield of the residue field, and suppose that there is a partial section $f:E\longrightarrow K$. If $F/E$ is a separably generated subextension of $Kv/E$ then we may extend $f$ to a partial section $F\longrightarrow K$.
\end{lemma}

\begin{proof}
Write $L_{1}:=f(E)$. Let $T$ be a separating transcendence base for $F/E$ and, for each $t\in T$, choose $s_{t}\in K$ such that $s_{t}v=t$. Then $S:=\{s_{t}\;|\;t\in T\}$ is algebraically independent over $L_{1}$. Thus we may extend $f$ to a partial section $E(T)\longrightarrow L_{1}(S)$ by sending $t\longmapsto s_{t}$.

Let $L_{2}$ be the relative separable algebraic closure of $L_{1}(S)$ in $K$. By Hensel's Lemma, $L_{2}v$ is separably algebraically closed in $Kv$. Thus $F$ is contained in $L_{2}v$. Since $v$ is trivial on $L_{2}$, the restriction of the residue map to $L_{2}$ is an isomorphism $L_{2}\longrightarrow L_{2}v$. Thus the restriction to $F$ of the inverse of the residue map is a partial section $F\longrightarrow K$ which extends $f$, as required.
\end{proof}

Recall that a valued field $(K,v)$ of residue characteristic $p$ is {\em tame} if it is henselian,
the value group $vK$ is $p$-divisible,
the residue field $Kv$ is perfect,
and $(K,v)$ is defectless, i.e.~for every finite extension $L/K$,
$$
 [L:K]=[Lv:Kv]\cdot[vL:vK].
$$

\begin{proposition}\label{fact:tamification}
Let $(K,v)$ be a valued field. There exists an extension $(K^{t},v^t)$ of $(K,v)$ such that $(K^t,v^t)$ is tame,
$K^t$ is perfect, $v^tK^{t}=\frac{1}{p^{\infty}}vK$, and $K^{t}v^t=Kv^{\mathrm{perf}}$.
\end{proposition}

\begin{proof}
In the special case ${\rm char}(K)={\rm char}(Kv)$, any maximal immediate extension of $K^{\rm perf}$ satisfies the claim.
In general, \cite[Theorem 2.1, Proposition 4.1, and Proposition 4.5(i)]{Kuhlmann-Pank-Roquette86} gives such a $K^t$ that is in addition algebraic over $K$.
\end{proof}

\section{Model theory of valued fields}

\noindent
Let 
$$
 \mathcal{L}_{\mathrm{ring}}=\{+,-,\cdot,0,1\}
$$ 
be the language of rings and let 
$$
 \mathcal{L}_{\rm vf}=\{+^K,-^K,\cdot^K,0^K,1^K,+^k,-^k,\cdot^k,0^k,1^k,+^\Gamma,<^\Gamma,0^\Gamma,\infty^\Gamma,v,{\rm res}\}
$$ 
be a three sorted language for valued fields
(like the Denef-Pas language, but without an angular component)
with a sort $K$ for the field itself, a sort $\Gamma\cup\{\infty\}$ for the value group with infinity, and a sort $k$ for the residue field,
as well as both the valuation map $v$ and the residue map ${\rm res}$,
which we interpret as the constant $0^k$ map outside the valuation ring. 
For a field $C$, we let $\mathcal{L}_{\mathrm{ring}}(C)$ and $\mathcal{L}_{\rm vf}(C)$ be the languages obtained by adding symbols for elements of $C$. In the case of $\mathcal{L}_{\rm vf}(C)$, the constant symbols are added to the field sort $K$.

A valued field $(K,v)$ gives rise in the usual way to an $\mathcal{L}_{\rm vf}$-structure 
$$
 (K,vK\cup\{\infty\},Kv,v,\mathrm{res}),
$$
where $vK$ is the value group, $Kv$ is the residue field, and $\mathrm{res}$ is the residue map. For notational simplicity, we will usually write $(K,v)$ to refer to the $\mathcal{L}_{\rm vf}$-structure it induces. For further notational simplicity, we write $(K,D)$ instead of $(K,(d_{c})_{c\in C})$, where $D=\{d_{c}|c\in C\}$ is the set of interpretations of the constant symbols. Combining these two simplifications, we write $(K,v,D)$ for the $\mathcal{L}_{\rm vf}(C)$-structure 
$$
 (K,vK\cup\{\infty\},Kv,v,\mathrm{res},(d_{c})_{c\in C}).
$$ 
We also write $Dv$ for the set of residues of elements from $D$.

As usual, we say that an $\mathcal{L}_{\rm vf}(C)$-formula is an \em $\exists$-formula \rm if it is logically equivalent
to a formula in prenex normal form with only existential quantifiers (over any of the three sorts).
We say that an $\mathcal{L}_{\rm vf}(C)$-sentence is an \em $\forall^{k}\exists$-sentence \rm if it is logically equivalent to 
a sentence of the form $\forall\mathbf{x}\,\psi(\mathbf{x})$,
where 
$\psi$ is an $\exists$-formula and
the universal quantifiers range over the residue field sort.

Let $(K,v,D)\subseteq(L,w,E)$ be an extension of $\mathcal{L}_{\rm vf}(C)$-structures.
Note that $d_{c}=e_{c}$, for all $c\in C$.
We say that certain $\mathcal{L}_{\rm vf}(C)$-sentences $\phi$ \em go up \rm from $K$ to $L$ if $(K,v,D)\models\phi$
implies that $(L,w,E)\models\phi$.
For examples, $\exists$-sentences always go up every extension. Furthermore, if $(L,w)/(K,v)$ is an extension of valued fields such that $Lw/Kv$ is trivial, then $\forall^{k}\exists$-$\mathcal{L}_{\rm vf}(K)$-sentences go up from $(K,v)$ to $(L,w)$. Although the previous statement is not referenced directly, it underlies many of the arguments in Section \ref{sec:transfer}.

\begin{lemma}\label{lem:ex.closed.perf}
Let $L/K$ be an extension of fields.
If $K\preceq_\exists L$, then $K^{\rm perf}\preceq_\exists L^{\rm perf}$.
\end{lemma}

\begin{proof}
This is clear, since $K^{\rm perf}=\bigcup_n K^{p^{-n}}$ and $L^{\rm perf}=\bigcup_n L^{p^{-n}}$,
and the Frobenius gives that $K^{p^{-n}}\preceq_\exists L^{p^{-n}}$ for all $n$.
\end{proof}

In \cite{Kuhlmann??}, F.-V. Kuhlmann studies the model theory of tame fields:

\begin{proposition}\label{fact:relative.embedding.tame}
The elementary class of tame fields has the Relative Embedding Property. I.e. for tame fields $(K,v)$ and $(L,w)$ with common subfield $(F,u)$, if
\begin{enumerate}
\item $(F,u)$ is defectless,
\item $(L,w)$ is $|K|^{+}$-saturated,
\item $vK/uF$ is torsion-free and $Kv/Fu$ is separable, and
\item there are embeddings $\rho:vK\longrightarrow wL$ (over $uF$) and $\sigma:Kv\longrightarrow Lw$ (over $Fu$);
\end{enumerate}
then there exists an embedding $\iota:(K,v)\longrightarrow(L,w)$ over $(F,u)$ which respects $\rho$ and $\sigma$.
\end{proposition}

\begin{proof}
See \cite[Theorem 7.1]{Kuhlmann??}.
(Note that this result is stated in the language 
$$
 \mathcal{L}_{\rm vf}'=\{+,-,\cdot,{}^{-1},0,1,O\},
$$ 
where $O$ is a binary predicate which is interpreted in a valued field $(K,v)$ so that $O(a,b)$ if and only if $va\geq vb$. However, the exact choice of language does not directly affect us.)
\end{proof}

From \autoref{fact:relative.embedding.tame}, Kuhlmann deduces the following AKE principle:

\begin{theorem}\label{fact:AKE.tame}
The class of tame fields is an 
$\mathrm{AKE}^{\preceq}$-class:
If $(L,w)/(K,v)$ is an extension of tame fields with $vK\preceq wL$ and $Kv\preceq Lw$, then
$(K,v)\preceq(L,w)$.
\end{theorem}

\begin{proof}
See \cite[Theorem 1.4]{Kuhlmann??}.
\end{proof}

\section{Power series fields}

\noindent
For a field $F$ and an ordered abelian group $\Gamma$ we denote by $F((\Gamma))$ the field
of generalized power series with coefficients in $F$ and exponents in $\Gamma$, see e.g.~\cite[\S4.2]{Efrat}.
We identify $F((\mathbb{Z}))$ with the field of formal power series $F((t))$ 
and denote the power series valuation on any subfield of any $F((\Gamma))$ by $v_t$.

\begin{lemma}\label{lem:Hahn.tame}\label{lem:Hahn.maximal}
A field $(F((\Gamma)),v_{t})$ of generalized power series is maximal.
In particular, it is tame if and only if $F$ is perfect and $\Gamma$ is $p$-divisible.
\end{lemma}

\begin{proof}
See \cite[Theorem 18.4.1]{Efrat}
and note that maximal implies henselian and defectless. 
\end{proof}

\begin{proposition}\label{fact:Serre}
Let $A$ be a complete discrete (i.e. with value group $\mathbb{Z}$) equicharacteristic valuation ring. Let $F\subseteq A$ be a set of representatives for the residue classes which forms a field. Let $s\in A$ be a uniformiser (i.e. an element of least positive value). Then $A$ is isomorphic to $F[[s]]$ by an isomorphism which fixes $F$ pointwise.
\end{proposition}

\begin{proof}
See \cite{Serre79}, Ch.~2 Prop.~5 and discussion following the example.
\end{proof}

\begin{corollary}\label{prp:finite.extensions.of.power.series}
Let $F$ be a field and let $E/F((t))$ be a finite extension such that $Ev_t=F$. 
Then $(E,v_t,F)$ is isomorphic to $(F((s)),v_s,F)$. 
This applies in particular to finite extensions of $F((t))$ inside $F((\mathbb{Q}))$.
\end{corollary}

\begin{proof}
We are already provided with a section since $F\subseteq F((t))\subseteq E$ and $Ev_t=F$. 
Since $E/F((t))$ is finite, $E$ is also a complete discrete equicharacteristic valued field (cf.~\cite[Ch.~2 Prop.~3]{Serre79}). By \autoref{fact:Serre}, there is an $F$-isomorphism of valued fields $E\longrightarrow F((s))$.
\end{proof}

\begin{definition}
We denote by 
$F(t)^h$ the henselization of $F(t)$ with respect to $v_t$, i.e.\ the relative algebraic closure of $F(t)$ in $F((t))$,
and by $F((t))^{\mathbb{Q}}$ the relative algebraic closure of $F((t))$ in $F((\mathbb{Q}))$.
\end{definition}

\begin{lemma}\label{lem:F(t)h.ex.closed}
For any field $F$ we have $(F(t)^{h},v_{t})\preceq_{\exists}(F((t)),v_{t})$.
\end{lemma}

\begin{proof}
See \cite[Theorem 5.12]{Kuhlmann??}.
\end{proof}

The following proposition may be deduced from the more general \cite[Lemma 3.7]{Kuhlmann??}, but we give a proof in this special case for the convenience of the reader.

\begin{proposition}\label{lem:racl.tame}
If $F$ is perfect, then $F((t))^{\mathbb{Q}}$ is tame.
\end{proposition}

\begin{proof}
We have that $F((t))^{\mathbb{Q}}v_t=F$ is perfect and $v_tF((t))^{\mathbb{Q}}=\mathbb{Q}$ is $p$-divisible.
Moreover, as an algebraic extension of the henselian field $F((t))$, $F((t))^{\mathbb{Q}}$ is henselian.
It remains to show that $F((t))^{\mathbb{Q}}$ is defectless. 

Let $E/F((t))^{\mathbb{Q}}$ be a finite extension of degree $n$. 
Since 
$F((\mathbb{Q}))$ is perfect, so is $F((t))^{\mathbb{Q}}$,
hence $F((\mathbb{Q}))/F((t))^{\mathbb{Q}}$ is regular.
Therefore, if we denote by $E'=F((\mathbb{Q}))\cdot E$ the compositum of $F((\mathbb{Q}))$ and $E$ in an algebraic closure of $F((\mathbb{Q}))$,
then $[E':F((\mathbb{Q}))]=n$.
Since $F((\mathbb{Q}))$ is maximal (\autoref{lem:Hahn.maximal}), 
$E'/F((\mathbb{Q}))$ is defectless. 
So since $(F((\mathbb{Q})),v_t)$ is henselian and $v_tF((\mathbb{Q}))=\mathbb{Q}$ is divisible,
we get that $[E'v_t:F]=n$.
Since $E'v_t/F$ is separable, we can assume without loss of generality that $F':=E'v_t\subseteq E'$ 
(\autoref{lem:section}).
$$
 \xymatrix{
  F((\mathbb{Q}))\ar@{-}[r]^{\quad n} & E'\ar@{=}[r] & E' \\
  F((t))^{\mathbb{Q}}\ar@{-}[r]^{\quad n}\ar@{-}[u]^{\rm reg.} & E\ar@{-}[r]\ar@{-}[u]^{\rm reg.} &  EF'\ar@{-}[u]\\
  F\ar@{-}[rr]^n\ar@{-}[u] && F'\ar@{-}[u]
 }
$$
The extension $E'/E$ is also regular, since $E/F((t))^{\mathbb{Q}}$ is algebraic. 
In particular, $E$ is relatively algebraically closed in $E'$; so since $EF'/E$ is algebraic we have that $F'\subseteq E$. Thus $Ev_t=F'$,
which shows that $E/F((t))^{\mathbb{Q}}$ is defectless.
\end{proof}

In particular, \autoref{fact:AKE.tame} implies that $F((t))^\mathbb{Q}\preceq F((\mathbb{Q}))$.
We therefore get the following picture:
$$
\xymatrix{
 F(t)\ar@{-}[r]^{\rm alg.} & F(t)^h\ar@{-}[r]^{\preceq_\exists} & F((t))\ar@{-}[r]^{\rm alg.} & F((t))^\mathbb{Q}\ar@{-}[r]^{\preceq} & F((\mathbb{Q}))
}
$$

\section{The transfer of universal-existential sentences}
\label{sec:transfer}

\noindent
Throughout this section $F/C$ will be a separable extension of fields of characteristic $p$.
We show that the truth of $\forall^{k}\exists$-sentences transfers between various valued fields. Usually the valued fields considered will have only elementarily equivalent residue fields. However, for convenience, we will sometimes discuss $\exists$-sentences with additional parameters from the residue field. 

\begin{lemma}\label{lem:go.down.from.power.series}\rm(\bf Going down from $\mathbf{F}((\Gamma))$\rm) \em
Suppose that $F$ is perfect. Let $\phi$ be an $\exists$-$\mathcal{L}_{\rm vf}(F)$-sentence, let $F\preceq\mathbf{F}$ be an elementary extension, and let $\Gamma$ be an ordered abelian group. If $(\mathbf{F}((\Gamma)),v_{t},F)\models\phi$, then $(F(t)^{h},v_{t},F)\models\phi$.
\end{lemma}

\begin{proof}
Without loss of generality we may assume that $\Gamma$ is nontrivial. 
For notational simplicity, we will suppress the parameters $F$ from the notation. 
Let $\Delta$ be the divisible hull of $\Gamma$. 
Then $(\mathbf{F}((\Gamma)),v_{t})\subseteq(\mathbf{F}((\Delta)),v_{t})$, and existential sentences `go up', so $(\mathbf{F}((\Delta)),v_{t})\models\phi$. 

Choose an embedding of $\mathbb{Q}$ into $\Delta$;
this induces an embedding $(F((\mathbb{Q})),v_t)\subseteq(\mathbf{F}((\Delta)),v_t)$,
and therefore $(F((t))^{\mathbb{Q}},v_{t})\subseteq(\mathbf{F}((\Delta)),v_{t})$. 
Since the theory of divisible ordered abelian groups is model complete (see e.g.~\cite[Thm.~4.1.1]{PrestelDelzell}), 
$$
 v_tF((t))^{\mathbb{Q}}=\mathbb{Q}\preceq\Delta=v_t\mathbf{F}((\Delta)).
$$
Moreover,
$$
 F((t))^{\mathbb{Q}}v_t=F\preceq\mathbf{F}=\mathbf{F}((\Delta))v_t.
$$ 
Thus, since $(F((t))^{\mathbb{Q}},v_t)$ is tame by \autoref{lem:racl.tame} and $(\mathbf{F}((\Delta)),v_t)$ is tame by \autoref{lem:Hahn.tame},
\autoref{fact:AKE.tame} implies that
$$
 (F((t))^{\mathbb{Q}},v_{t})\preceq(\mathbf{F}((\Delta)),v_{t}).
$$
Therefore, $(F((t))^{\mathbb{Q}},v_{t})\models\phi$. 

Let $E$ be a finite extension of $F((t))$ that contains witnesses to the truth of $\phi$ in $(F((t))^{\mathbb{Q}},v_{t})$. 
Thus $(E,v_{t})\models\phi$. By \autoref{prp:finite.extensions.of.power.series}, there is an $\mathcal{L}_{\rm vf}(F)$-isomorphism 
$$
 f:(E,v_{t})\longrightarrow(F((t)),v_{t}).
$$
Thus $(F((t)),v_{t})\models\phi$. 
By \autoref{lem:F(t)h.ex.closed}, 
$$
 (F(t)^{h},v_{t})\preceq_{\exists}(F((t)),v_{t}),
$$
hence $(F(t)^{h},v_{t})\models\phi$, as claimed.
\end{proof}

\begin{definition}\label{def:EH(F/C)}
Let $\mathbf{H}(F/C)$ be the class of tuples $(K,v,D,i)$, where $(K,v,D)$ is an $\mathcal{L}_{\rm vf}(C)$-structure 
and $i:F\rightarrow Kv$ is a map such that
\begin{enumerate}
\item $(K,v)$ is an equicharacteristic henselian nontrivially valued field,
\item $c\longmapsto d_{c}$ is an $\mathcal{L}_{\mathrm{ring}}$-embedding $C\longrightarrow K$,
\item the valuation is trivial on $D$, and
\item $i:(F,C)\longrightarrow(Kv,Dv)$ is an $\mathcal{L}_{\mathrm{ring}}(C)$-embedding.
\end{enumerate}
\end{definition}

\begin{lemma}\label{lem:go.up.from.F(t)h}\rm(\bf Going up from $F(t)^{h}$\rm) \em
Let $\phi$ be an $\exists$-$\mathcal{L}_{\rm vf}$-sentence with parameters from $C$ and the residue sort of $(F(t)^{h},v_{t})$, and suppose that $(F(t)^{h},v_{t},C)\models\phi$. Then, for all $(K,v,D,i)\in\mathbf{H}(F/C)$, we have that $(K,v,D)\models\phi$ (where we replace the parameters from the residue sort by their images under the map $i$).
\end{lemma}

\begin{proof}
Write $\phi=\exists\mathbf{x}\;\psi(\mathbf{x};\mathbf{c},\bfbeta)$ for some quantifier-free formula $\psi$ and parameters $\mathbf{c}$ from $C$ and $\bfbeta$ from $F(t)^{h}v_{t}$. Note that the variables in the tuple $\mathbf{x}$ may be from any sorts. Let $\mathbf{a}$ be such that 
$$
 (F(t)^{h},v_{t},C)\models\psi(\mathbf{a};\mathbf{c};\bfbeta).
$$ 
Since $F(t)^{h}$ is the directed union of fields $E_0(t)^{h}$ for finitely generated subfields $E_0$ of $F$, there exists a subfield $E$ of $F$ 
containing $C$ such that $E/C$ is finitely generated, $\mathbf{a}\in E(t)^{h}$, and $\bfbeta\in E(t)^{h}v_{t}$. Thus 
$$
 (E(t)^{h},v_{t},C)\models\psi(\mathbf{a};\mathbf{c};\bfbeta).
$$ 

Since $F/C$ is separable and $E/C$ is finitely generated, $E$ is separably generated over $C$. Thus $i(E)/Dv$ is separably generated. 
Note that the map $Dv\longrightarrow D$ given by $d_{c}v\longmapsto d_{c}$ is a partial section. 
By \autoref{lem:section} we may extend it to a partial section $g:i(E)\longrightarrow K$. 
Let $h:=g\circ i|_E$ be the composition. Then 
$$
 h:(E,v_{0},C)\longrightarrow(K,v,D)
$$ 
is an $\mathcal{L}_{\rm vf}(C)$-embedding, where $v_{0}$ denotes the trivial valuation on $E$:
$$
 \xymatrix{
   & Kv & K\ar@{.>}[l]_{\rm res} \\
 F\ar@{->}[r]^{i} & i(F)\ar@{-}[u] \\
 E\ar@{-}[u]\ar@{->}[r]^{i|_E} & i(E)\ar@{-}[u]\ar@{->}[r]^{g} & h(E)\ar@{-}[uu] \\
 C\ar@{-}[u]\ar@{->}[r]^{\cong} & Dv\ar@{-}[u]\ar@{->}[r]^{\cong} & D\ar@{-}[u]
 }
$$
Since $(K,v)$ is nontrivial, there exists $s\in K^\times$ with $v(s)>0$, % \in\mathcal{M}(K)\setminus\{0\}$,
which must be transcendental over $h(E)$, since $v$ is trivial on $h(E)$. 
As the rational function field $E(t)$ admits (up to equivalence) only one valuation which is
trivial on $E$ and positive on $t$,
we may extend $h$ to an $\mathcal{L}_{\rm vf}(C)$-embedding 
$$
 h':(E(t),v_{t},C)\longrightarrow(K,v,D)
$$ 
by sending $t\longmapsto s$.
Since $(K,v)$ is henselian, there is a unique extension of $h'$ to an $\mathcal{L}_{\rm vf}(C)$-embedding 
$$
 h'':(E(t)^{h},v_{t},C)\longrightarrow(K,v,D).
$$ 
So, since existential sentences `go up',
$$
 (K,v,D)\models\psi(h''(\mathbf{a});h''(\mathbf{c});h''(\bfbeta));
$$ 
and thus $(K,v,D)\models\phi$, as claimed.
\end{proof}

\begin{definition}
We let $R_{F/C}$ be the $\mathcal{L}_{\mathrm{ring}}(C)$-theory of $F$ and let $R_{F/C}^{1}$ be the subtheory consisting of existential and universal sentences. Let $\mathbf{T}_{F/C}$ (respectively, $\mathbf{T}_{F/C}^{1}$) be the $\mathcal{L}_{\rm vf}(C)$-theory consisting of the following axioms (expressed informally about a structure $(K,v,D)$):
\begin{enumerate}
\item $(K,v)$ is an equicharacteristic henselian nontrivially valued field,
\item $c\longmapsto d_{c}$ is an $\mathcal{L}_{\mathrm{ring}}$-embedding $C\longrightarrow K$,
\item the valuation $v$ is trivial on $D$, and
\item $(Kv,Dv)$ is a model of $R_{F/C}$ (respectively, $R_{F/C}^{1}$).
\end{enumerate}
\end{definition}

The `1' is intended to suggest that the sentences considered contain only one type of quantifier. Note that for any $(K,v,D)\models\mathbf{T}_{F/C}^{1}$, $d_{c}v\longmapsto d_{c}$ is a partial section of the residue map.
Let $\phi$ be an $\forall^{k}\exists$-sentence and write $\phi=\forall^{k}\mathbf{x}\;\psi(\mathbf{x})$ for some $\exists$-formula $\psi(\mathbf{x})$ with free variables $\mathbf{x}$ belonging to the residue field sort. Let ${}^{\mathbf{x}}Kv$ denote the set of $\mathbf{x}$-tuples from $Kv$. Then we observe that $(K,v,D)\models\phi$ if and only if ${}^{\mathbf{x}}Kv\subseteq\psi(K)$.
In this next proposition we show that, roughly: if $\mathbf{T}_{F/C}$ is consistent with the property `${}^{\mathbf{x}}F\subseteq\psi$' then in fact $\mathbf{T}_{F^{\mathrm{perf}}/C^{\mathrm{perf}}}$ entails `${}^{\mathbf{x}}F\subseteq\psi$'.

\begin{proposition}\label{prp:AkE.transfer}\rm(\bf Main Proposition\rm) \em
Let $\psi(\mathbf{x})$ be an $\exists$-$\mathcal{L}_{\rm vf}(C)$-formula with free variables $\mathbf{x}$ belonging to the residue field sort. 
Suppose there exists $(K,v,D)\models\mathbf{T}_{F/C}\cup\{\forall^{k}\mathbf{x}\;\psi(\mathbf{x})\}$.
Then, for all $(L,w,E,i)\in\mathbf{H}(F^{\mathrm{perf}}/C^{\mathrm{perf}})$, we have ${}^{\mathbf{x}}i(F)\subseteq\psi(L)$.
\end{proposition}
\begin{proof}
Since $(K,v,D)$ models $\mathbf{T}_{F/C}$, we have $(Kv,Dv)\equiv(F,C)$. By passing, if necessary, to an elementary extension of $(K,v,D)$, there is an elementary embedding
$$
 f:(F,C)\overset{\preceq}{\longrightarrow}(Kv,Dv).
$$
As noted after the definition of $\mathbf{T}_{F/C}$, the map $g_0:Dv\longrightarrow D$ given by $d_{c}v\longrightarrow d_{c}$ is a partial section. Since $F/C$ is separable, $f(F)/Dv$ is also separable. 
Thus any finitely generated subextension of $f(F)/Dv$ is separably generated. 
By \autoref{lem:section} we may pass again - if necessary - to an elementary extension and extend $g_0$ to a partial section $g:f(F)\longrightarrow K$. Note that $g$ is also an $\mathcal{L}_{\mathrm{ring}}(C)$-embedding $(f(F),Dv)\longrightarrow(K,D)$.

Let $h:=g\circ f$. Then $h:(F,C)\longrightarrow(K,D)$ is an $\mathcal{L}_{\mathrm{ring}}(C)$-embedding. Because $g$ is a section, the valuation $v$ is trivial when restricted to the image of $h$. Thus, if $v_{0}$ denotes the trivial valuation on $F$, the map $h$ is an $\mathcal{L}_{\rm vf}(C)$-embedding $(F,v_{0},C)\longrightarrow(K,v,D)$. The induced embedding of residue fields $\bar{h}:Fv_{0}\longrightarrow Kv$ is the composition of the elementary embedding $f$ with an isomorphism. Thus $\bar{h}:Fv_{0}\longrightarrow Kv$ is an elementary embedding.
From now on we identify $(F,v_{0},C)$ with its image under $h$ as a substructure of $(K,v,D)$, noting that the residue field extension is an elementary extension.
$$
 \xymatrix{
  & Kv & K\ar@{.>}[l]_{\rm res} \\
  F\ar@{->}[r]^f & f(F)\ar@{-}[u]\ar@{->}[r]^{g} & h(F)\ar@{-}[u] \\
  C\ar@{-}[u]\ar@{->}[r]^{\cong} & Dv\ar@{-}[u]^{\rm sep}\ar@{->}[r]^{g_0} & D\ar@{-}[u] 
 }
$$

Choose an extension $(K^{t},v^t)/(K,v)$ as in \autoref{fact:tamification}. 
Since $K^{t}$ is perfect, we can embed $D^{\mathrm{perf}}$ into $K^{t}$ over $D$ so that $(K^{t},v^t,D^{\mathrm{perf}})$ is an $\mathcal{L}_{\rm vf}(C^{\mathrm{perf}})$-structure. Furthermore $(F^{\mathrm{perf}},v_{0},C^{\mathrm{perf}})$ is naturally (identified with) a substructure of $(K^{t},v^t,D^{\mathrm{perf}})$. 
Since $Fv_{0}\preceq Kv$, \autoref{lem:ex.closed.perf} gives that
$$
 F^{\mathrm{perf}}v_{0}=Fv_{0}^{\rm perf}\preceq_\exists Kv^{\rm perf}=K^{t}v^t.
$$
Thus there is an elementary extension $F^{\mathrm{perf}}v_{0}\preceq\mathbf{F}$ and an embedding $\sigma:K^{t}v^t\longrightarrow\mathbf{F}$ over $F^{\mathrm{perf}}v_{0}$,
see the diagram below.

Now we consider the two valued fields $(K^{t},v^t)$ and $(\mathbf{F}((v^tK^{t})),v_{t})$ with common subfield $(F^{\mathrm{perf}},v_{0})$. 
Note that $K^t$ is tame by definition, and $\mathbf{F}((v^tK^{t}))$ is tame by \autoref{lem:Hahn.tame}.
As a trivially valued field, $(F^{\mathrm{perf}},v_{0})$ is defectless. 
The extension of value groups $v^tK^{t}/v_{0}F^{\mathrm{perf}}$ is isomorphic to $v^tK^{t}$, thus it is torsion-free. 
The extension $K^{t}v^t/F^{\mathrm{perf}}v_{0}$ is separable since $F^{\mathrm{perf}}v_{0}$ is isomorphic to $F^{\mathrm{perf}}$ which is perfect.
Let $(\mathbf{F}((v^tK^{t})),v_{t})^{*}$ be a $|K|^{+}$-saturated elementary extension of $(\mathbf{F}((v^tK^{t})),v_{t})$. We have satisfied the hypotheses of \autoref{fact:relative.embedding.tame}, thus there exists an embedding 
$$
 \iota:(K^{t},v^t)\longrightarrow(\mathbf{F}((v^tK^{t})),v_{t})^{*}
$$ 
over $(F^{\mathrm{perf}},v_{0})$. 
Since existential sentences `go up', we get that 
$(\mathbf{F}((v^tK^{t})),v_{t})^{*}$, and therefore also $(\mathbf{F}((v^tK^{t})),v_{t})$, 
models the existential $\mathcal{L}_{\rm vf}(F^{\mathrm{perf}})$-theory of $(K^{t},v^t)$.
$$
 \xymatrix{
  & & & \mathbf{F}((v^tK^t))^\ast \\
  & & & \mathbf{F}((v^tK^t))\ar@{-}[u]_{\preceq} \\
  & K^t\ar@{->}[rruu]^{\iota}\ar@{.>}[r]_{\rm res} & K^tv^t\ar@{->}[r]^{\sigma} & \mathbf{F}\ar@{-}[u] \\
  K\ar@{-}[r] & K^{\rm perf}\ar@{-}[u] & & \\
  F\ar@{-}[u]\ar@{-}[r] & F^{\rm perf}\ar@{-}[u]\ar@{->}[r]^{\cong} & F^{\rm perf}v_0\ar@{-}[uu]^{\preceq_\exists}\ar@{-}[ruu]_{\preceq} & \\
  C\ar@{-}[u]\ar@{-}[r] & C^{\rm perf}\ar@{-}[u] & & 
 }
$$
Our assumption was that $\psi(\mathbf{x})$ is an $\exists$-$\mathcal{L}_{\rm vf}(C)$-formula with free variables $\mathbf{x}$ belonging to the residue field sort, and that $(K,v,D)\models\forall^{k}\mathbf{x}\;\psi(\mathbf{x})$, i.e. ${}^{\mathbf{x}}Kv\subseteq\psi(K)$. Then ${}^{\mathbf{x}}Fv\subseteq{}^{\mathbf{x}}Kv\subseteq\psi(K)$ (note that we write $Fv$ rather than $F$ because we have identified $F$ with a subfield of $K$). 
Let 
$$
 \Psi_{F}:=\{\psi(\mathbf{a})\;|\;\mathbf{a}\in{}^{\mathbf{x}}Fv\}.
$$
Then $\Psi_{F}$ is a set of $\exists$-$\mathcal{L}_{\rm vf}(C)$-sentences (with additional parameters from $Fv$) which is equivalent to the property that `${}^{\mathbf{x}}Fv\subseteq\psi$'. We may now restate our assumption as $(K,v)\models\Psi_{F}$. Since existential sentences `go up', $(K^{t},v^t)\models\Psi_{F}$. By the result of the previous paragraph, we have $(\mathbf{F}((v^tK^{t})),v_{t})\models\Psi_{F}$. By an application of \autoref{lem:go.down.from.power.series}, $(F^{\mathrm{perf}}(t)^{h},v_{t})\models\Psi_{F}$. By \autoref{lem:go.up.from.F(t)h}, $(L,w)\models\Psi_{F}$ (where we replace the parameters from $Fv$ by their images under the map $i$). 
This shows that ${}^{\mathbf{x}}i(F)\subseteq\psi(L)$, as claimed.
\end{proof}

\begin{corollary}\label{cor:near.AkE.completeness}\rm(\bf Near $\forall^{k}\exists$-$C$-completeness\rm) \em
Let $\psi(\mathbf{x})$ be an $\exists$-$\mathcal{L}_{\rm vf}(C)$-formula with free variables $\mathbf{x}$ belonging to the residue field sort. 
Suppose there exists $(K,v,D)\models\mathbf{T}_{F/C}\cup\{\forall^{k}\mathbf{x}\;\psi(\mathbf{x})\}$. 
Then there exists $n\in\mathbb{N}$ such that, for all $(L,w,E)\models\mathbf{T}_{F/C}$, we have ${}^{\mathbf{x}}Lw\subseteq\psi(L^{p^{-n}})$.
\end{corollary}
\begin{proof}
Let $(L,w,E)\models\mathbf{T}_{F/C}$. 
As $F/C$ is separable and $(Lw,Ew)\equiv(F,C)$ as $\mathcal{L}_{\rm ring}(C)$-structures,
also $Lw/Ew$ is separable.
In particular,
$(K,v,D),(L,w,E)\models\mathbf{T}_{Lw/Ew}$ and we may apply the conclusion of \autoref{prp:AkE.transfer} to 
$$
 (L^{\mathrm{perf}},w,E^{\mathrm{perf}},{\rm id})\in\mathbf{H}(Lw^{\mathrm{perf}}/Ew^{\mathrm{perf}}).
$$
Thus we have that ${}^{\mathbf{x}}Lw\subseteq\psi(L^{\mathrm{perf}})$. To find $n$, we use a simple compactness argument, as follows.

Write the formula $\psi(\mathbf{x})$ as $\exists\mathbf{y}\;\rho(\mathbf{x},\mathbf{y},\mathbf{c})$, for a quantifier-free $\mathcal{L}_{\rm vf}$-formula $\rho$. For each $n\in\mathbb{N}$, let $\psi_{n}(\mathbf{x})$ be the formula $\exists\mathbf{y}\;\rho(\mathbf{x}^{p^{n}},\mathbf{y},\mathbf{c}^{p^{n}})$ and consider the $\mathcal{L}_{vf}(C)$-structure $(L^{p^{-n}},w,E)$ which extends $(L,w,E)$. Then, for $\mathbf{a}\in{}^{\mathbf{x}}Lw$, $\mathbf{a}\in\psi(L^{p^{-n}})$ if and only if $\mathbf{a}\in\psi_{n}(L)$. Let $p(\mathbf{x})$ be the set of formulas $\{\neg\psi_{n}(\mathbf{x})\;|\;n\in\mathbb{N}\}$. If $p(\mathbf{x})$ is a type, i.e. $p(\mathbf{x})$ is consistent with $\mathbf{T}_{F/C}$, then we may realise it by a tuple $\mathbf{a}$ in a model $(L,w,E)\models\mathbf{T}_{F/C}$. Thus $\mathbf{a}\notin\psi(L^{p^{-n}})$, for all $n\in\mathbb{N}$. Since $L^{\mathrm{perf}}$ is the directed union $\bigcup_{n\in\mathbb{N}}L^{p^{-n}}$ (even as $\mathcal{L}_{\rm vf}(C)$-structures), we have that $\mathbf{a}\notin\psi(L^{\mathrm{perf}})$. This contradicts the result of the previous paragraph. 

Consequently, there exists $n\in\mathbb{N}$ such that $\mathbf{T}_{F/C}$ entails $\forall^{k}\mathbf{x}\;\psi_{n}(\mathbf{x})$. Equivalently, for all $(L,w,E)\models\mathbf{T}_{F/C}$, we have ${}^{\mathbf{x}}Lw\subseteq\psi(L^{p^{-n}})$, as required.
\end{proof}

\begin{corollary}\label{cor:residue.perfect}\rm(\bf Perfect residue field, $\forall^{k}\exists$-$C$-completeness\rm) \em
Suppose that $F$ is perfect. Then $\mathbf{T}_{F/C}$ is $\forall^{k}\exists$-$C$-complete, i.e. for any $\forall^{k}\exists$-$\mathcal{L}_{\rm vf}(C)$-sentence $\phi$, either $\mathbf{T}_{F/C}\models\phi$ or $\mathbf{T}_{F/C}\models\neg\phi$.
\end{corollary}
\begin{proof}
Suppose that there is $(K,v,D)\models\mathbf{T}_{F/C}\cup\{\phi\}$ and let $(L,w,E)\models\mathbf{T}_{F/C}$. Then $(K,v,D)\models\mathbf{T}_{Lw/Ew}$ and 
$$
 (L,w,E,{\rm id})\in\mathbf{H}(Lw/Ew)=\mathbf{H}(Lw^{\mathrm{perf}}/Ew^{\mathrm{perf}}).
$$ 
We write $\phi=\forall^{k}\mathbf{x}\;\psi(\mathbf{x})$ for some $\exists$-$\mathcal{L}_{\rm vf}(C)$-formula $\psi(\mathbf{x})$ with free variables $\mathbf{x}$ belonging to the residue field sort. Then $(K,v,D)\models\phi$ means that ${}^{\mathbf{x}}Kv\subseteq\psi(K)$. Applying \autoref{prp:AkE.transfer}, we have that ${}^{\mathbf{x}}Lw\subseteq\psi(L)$. Thus $(L,w,E)\models\phi$. This shows that $\mathbf{T}_{F/C}\models\phi$, as required.
\end{proof}

\begin{remark}
We do not know whether the assumption that $F$ is perfect is necessary in \autoref{cor:residue.perfect}.

However, note that \autoref{cor:residue.perfect} cannot be extended from $\forall^k\exists$-sentences to arbitrary
$\forall\exists$-sentences (even without parameters and with only one universal quantifier):
For example, the sentence
$$
 \forall x \exists y \; ( v(x)=v(y^2) )
$$
expresses $2$-divisibility of the value group, hence is satisfied in $F((\mathbb{Q}))$ but not in $F((t))$.

On the other hand, one could generalize \autoref{cor:residue.perfect} by slightly adapting the proof to allow also sentences with more general quantifiers over the residue field, namely $Q^k\exists$-$\mathcal{L}_{\rm vf}(C)$-sentences, i.e.\ sentences of the form
$$
 \exists^k \mathbf{x}_1\forall^k \mathbf{y}_1\dots\exists^k\mathbf{x}_n\forall^k\mathbf{y}_n\;\psi(\mathbf{x}_1,\mathbf{y}_1,\dots,\mathbf{x}_n,\mathbf{y}_n)
$$
with $\psi(\mathbf{x}_1,\mathbf{y}_1,\dots,\mathbf{x}_n,\mathbf{y}_n)$ an $\exists$-$\mathcal{L}_{\rm vf}(C)$-formula.
\end{remark}

\section{The existential theory}

\noindent
We now restrict the machinery of the previous section to existential sentences
and prove \autoref{thm:intro.E.complete} from the introduction.
We fix a field $F$, let $C$ be the prime field of $F$ and denote
$\mathbf{T}_F=\mathbf{T}_{F/C}$, $\mathbf{H}(F)=\mathbf{H}(F/C)$.

\begin{lemma}\label{lem:E.complete}
$\mathbf{T}_{F}$ is $\exists$-complete, i.e. for any $\exists$-$\mathcal{L}_{\rm vf}$-sentence $\phi$, either $\mathbf{T}_{F}\models\phi$ or $\mathbf{T}_{F}\models\neg\phi$.
\end{lemma}

\begin{proof}
Suppose that $\mathbf{T}_{F}\cup\{\phi\}$ is consistent. Thus there exists $(K,v)\models\mathbf{T}_{F}\cup\{\phi\}$. 
Simply viewing $\phi$ as an $\forall^{k}\exists$-formula $\forall^kx\,\psi(x)$
with $\psi(x)=\phi$, we have that $Kv\subseteq\psi(K)$. By \autoref{cor:near.AkE.completeness} there exists $n\in\mathbb{N}$ such that, for every $(L,w)\models\mathbf{T}_{F}$, $Lw\subseteq\psi(L^{p^{-n}})$. In particular, $\psi(L^{p^{-n}})$ is nonempty. Since no parameters appear in $\psi$, we may apply the $n$-th power of the Frobenius map to get that $\psi(L)$ is nonempty, for every $(L,w)\models\mathbf{T}_{F}$. 
Viewing $\phi$ as an $\exists$-sentence again, we have that $(L,w)\models\phi$. 
Thus $\mathbf{T}_{F}\models\phi$, as required.
\end{proof}

For the proof of \autoref{thm:intro.E.complete} it remains to show that $\mathbf{T}_{F}^{1}$ already entails those existential and universal sentences which are entailed by $\mathbf{T}_{F}$.

\begin{definition}
We define two subtheories of $\mathbf{T}_{F}^{1}$. 
Let $T_{F}^{\exists}$ be the $\mathcal{L}_{\rm vf}$-theory consisting of the following axioms (expressed informally about a structure $(K,v)$):
\begin{enumerate}
\item $(K,v)$ is an equicharacteristic henselian nontrivially valued field and
\item $Kv$ is a model of the existential $\mathcal{L}_{\rm ring}$-theory of $F$.
\end{enumerate}
Let $T_{F}^{\forall}$ be the $\mathcal{L}_{\rm vf}$-theory consisting of the following axioms (again expressed informally):
\begin{enumerate}
\item $(K,v)$ is an equicharacteristic henselian nontrivially valued field and
\item $Kv$ is a model of the universal $\mathcal{L}_{\rm ring}$-theory of $F$.
\end{enumerate}
\end{definition}
Note that %both $T_{F}^{\exists}$ and $T_{F}^{\forall}$ are $\mathcal{L}_{\rm vf}$-theories and that 
$\mathbf{T}_{F}^{1}\equiv T_{F}^{\exists}\cup T_{F}^{\forall}$.

\begin{lemma}\label{lem:existential.axiom}
Let $\phi$ be an existential $\mathcal{L}_{\rm vf}$-sentence. If $\mathbf{T}_{F}\models\phi$ then $T_{F}^{\exists}\models\phi$.
\end{lemma}

\begin{proof}
Let $(K,v)\models T_{F}^{\exists}$. 
Then $Kv$ is a model of $\mathrm{Th}_{\exists}(F)$; equivalently the theory of $Kv$ is consistent with the atomic diagram of $F$. 
Thus there is an elementary extension $(K,v)\preceq(K^{*},v^{*})$ with an embedding $\sigma:F\rightarrow K^{*}v^{*}$, cf.~\cite[Lemma 2.3.3]{Marker00}.
Note that $(K^{*},v^{*},\sigma)\in\mathbf{H}(F)$
and that $(F(t)^h,v_t)\models\mathbf{T}_F$, hence $(F(t)^h,v_t)\models\phi$.
Therefore, \autoref{lem:go.up.from.F(t)h} implies that $(K^{*},v^{*})\models\phi$; thus $(K,v)\models\phi$. This shows that $T_{F}^{\exists}\models\phi$.
\end{proof}

\begin{lemma}\label{lem:universal.axiom}
Let $\phi$ be a universal $\mathcal{L}_{\rm vf}$-sentence. If $\mathbf{T}_{F}\models\phi$ then $T_{F}^{\forall}\models\phi$.
\end{lemma}

\begin{proof}
Let $(K,v)\models T_{F}^{\forall}$. Then $Kv\models\mathrm{Th}_{\forall}(F)$. 
There exists $F'\equiv F$ with an embedding $\sigma:Kv\rightarrow F'$, see \cite[Ex.~2.5.10]{Marker00}. 
Using \autoref{fact:complete.the.square}, we may choose an equicharacteristic nontrivially valued field $(L,w)$ which extends $(K,v)$ and is such that $Lw$ is isomorphic to $F'$. In particular $Lw\equiv F$. Let $(L,w)^{h}$ be the henselisation of $(L,w)$; then we have $(L,w)^{h}\models\mathbf{T}_{F}$,
so $(L,w)^h\models\phi$. 
Since $\phi$ is universal, we conclude that $(K,v)\models\phi$.
\end{proof}

\begin{theorem}\label{thm:E.complete}\rm(\bf $\exists$-completeness\rm) \em 
$\mathbf{T}_{F}^{1}$ is $\exists$-complete, i.e. for any $\exists$-$\mathcal{L}_{\rm vf}$-sentence $\phi$ either $\mathbf{T}_{F}^{1}\models\phi$ or $\mathbf{T}_{F}^{1}\models\neg\phi$.
\end{theorem}

\begin{proof} %[Proof of \autoref{thm:E.complete}]
Let $\phi$ be an existential $\mathcal{L}_{\rm vf}$-sentence. By \autoref{lem:E.complete}, either $\mathbf{T}_{F}\models\phi$ or $\mathbf{T}_{F}\models\neg\phi$. In the first case we apply \autoref{lem:existential.axiom} and find that $T_{F}^{\exists}\models\phi$; in the second case we apply \autoref{lem:universal.axiom} and find that $T_{F}^{\forall}\models\neg\phi$. Since $\mathbf{T}_{F}^{1}\equiv T_{F}^{\exists}\cup T_{F}^{\forall}$, in either case $\mathbf{T}_{F}^{1}$ `decides' $\phi$, and we are done.
\end{proof}

\begin{remark}
Let $\chi(x)$ be an existential $\mathcal{L}_{\rm ring}$-formula with one free variable. 
In \cite{Anscombe-Fehm??} and other work on definable henselian valuations, we apply \autoref{thm:E.complete} to the following $\exists$- or $\forall$-$\mathcal{L}_{\rm vf}$-sentences.
\begin{enumerate}
\item $\forall x\;(\chi(x)\longrightarrow v(x)\geq0)$,
\item $\forall x\;(\chi(x)\longrightarrow v(x)>0)$, and
\item $\exists x\;(v(x)>0\wedge x\neq0\wedge\chi(x))$.
\end{enumerate}
We also apply \autoref{cor:near.AkE.completeness} to the $\forall^{k}\exists$-$\mathcal{L}_{\rm vf}$-sentence
\begin{enumerate}
\setcounter{enumi}{3}
\item $\forall^{k}x\exists y\;(\mathrm{res}(y)=x\wedge\chi(y))$.
\end{enumerate}
\end{remark}

\section{An `Existential AKE Principle' and existential decidability}
\label{sec:decidability}

\noindent
\autoref{thm:E.complete} shows that the existential (respectively, universal) theory of an equicharacteristic henselian nontrivially valued field only depends only on the existential (resp. universal) theory of its residue field. We formulate this in the following `Existential AKE Principle'.

\begin{theorem}\label{thm:existential.AKE}
Let $(K,v)$ and $(L,w)$ be equicharacteristic henselian nontrivially valued fields. Then
$$
 (K,v)\models\mathrm{Th}_{\exists}(L,w)\;\;\text{ if and only if }\;\;Kv\models\mathrm{Th}_{\exists}(Lw).
$$
\end{theorem}

\begin{proof}
$(\Longrightarrow)$ Note that the maximal ideal is defined by the quantifier-free formula $v(x)>0$.
Therefore any existential statement about the residue field can be translated into an existential statement about the valued field.

$(\Longleftarrow)$. If $Kv\models\mathrm{Th}_{\exists}(Lw)$ then $(K,v)\models T_{Lw}^{\exists}$. By \autoref{lem:E.complete}, $\mathbf{T}_{Lw}$ entails the existential theory of $(L,w)$; and, by \autoref{lem:existential.axiom}, $T_{Lw}^{\exists}$ entails the existential consequences of $\mathbf{T}_{Lw}$. Combining these two statements, we have that $T_{Lw}^{\exists}$ entails the existential theory of $(L,w)$. Therefore $(K,v)$ models the existential theory of $(L,w)$.
\end{proof}

\begin{corollary} %\rm(\bf Existential AKE Principle\rm) \em
\label{cor:existential.AKE}
Let $(K,v)$ and $(L,w)$ be equicharacteristic henselian nontrivially valued fields. Then
$$
 \mathrm{Th}_\exists(K,v)=\mathrm{Th}_{\exists}(L,w)\;\;\text{ if and only if }\;\;\mathrm{Th}_\exists(Kv)=\mathrm{Th}_{\exists}(Lw).
$$
\end{corollary}

\begin{proof}
This follows from \autoref{thm:existential.AKE},
since $\mathrm{Th}_\exists(K,v)=\mathrm{Th}_{\exists}(L,w)$ iff both
$(K,v)\models\mathrm{Th}_{\exists}(L,w)$ and $(L,w)\models\mathrm{Th}_{\exists}(K,v)$,
and $\mathrm{Th}_\exists(Kv)=\mathrm{Th}_{\exists}(Lw)$ iff both
$Kv\models\mathrm{Th}_{\exists}(Lw)$ and $Lw\models\mathrm{Th}_{\exists}(Kv)$.
\end{proof}

Note that \autoref{cor:existential.AKE} is in fact simply a reformulation of \autoref{thm:E.complete}.
Note moreover that, by the usual duality between existential and universal sentences, the same principle holds with `$\exists$' replaced by `$\forall$'.

\begin{remark}
The reader probably noticed that as opposed to the usual AKE principles, the value group does not occur here.
However, since all nontrivial ordered abelian groups have the same existential theory 
(which follows immediately from the completeness of the theory of divisible ordered abelian groups, see also \cite{GK}),
\autoref{cor:existential.AKE} could also be phrased as
$$
 \mathrm{Th}_\exists(K,v)=\mathrm{Th}_{\exists}(L,w)\;\text{ if and only if }\;\mathrm{Th}_\exists(Kv)=\mathrm{Th}_{\exists}(Lw)\text{ and }
 \mathrm{Th}_\exists(vK)=\mathrm{Th}_\exists(wL).
$$
In residue characteristic zero, this special form of the existential AKE principle was known before,
see e.g.~\cite[p.~192]{KoenigsmannSurvey}.
\end{remark}

Next we deduce \autoref{cor:intro.existential.decidability} from \autoref{thm:E.complete}.

\begin{corollary}\label{cor:existential.decidability}
Let $(K,v)$ be an equicharacteristic henselian nontrivially valued field. The following are equivalent.
\begin{enumerate}
\item ${\rm Th}_\exists(Kv)$ is decidable.
\item ${\rm Th}_\exists(K,v)$ is decidable.
\end{enumerate}
\end{corollary}

\begin{proof}[Proof of \autoref{cor:existential.decidability}]
$(2\implies1)$ As before, residue fields are interpreted in valued fields in such a way that existential statements about $Kv$ remain existential statements about $(K,v)$. 
Therefore, if $(K,v)$ is $\exists$-decidable, then $Kv$ is $\exists$-decidable.

$(1\implies2)$ Write $F:=Kv$ and suppose that $F$ is $\exists$-decidable. Then we may recursively enumerate the existential and universal theory $R^{1}_{F}$ of $F$. Consequently $\mathbf{T}_{F}^{1}$ is effectively axiomatisable. By \autoref{thm:E.complete}, $\mathbf{T}_{F}^{1}$ is an $\exists$-complete subtheory of $\mathrm{Th}(K,v)$. Thus we may decide the truth of existential (and universal) sentences in $(K,v)$.
\end{proof}

Let $\mathcal{L}_{\rm vf}(t)$ be the language of valued fields with an additional parameter $t$, and let $q$ be a prime power. 
In \cite{Denef-Schoutens03}, it is shown that resolution of singularities in characteristic $p$ would imply that the existential $\mathcal{L}_{\rm vf}(t)$-theory of $\mathbb{F}_{q}((t))$ is decidable. 
Using our methods we can prove the following weaker but unconditional result.

\begin{corollary}\label{cor:Fq((t))}
The existential theory of $\mathbb{F}_{q}((t))$ in the language of valued fields is decidable.
\end{corollary}

\begin{proof}[First proof]
We can apply \autoref{cor:existential.decidability}, noting that ${\rm Th}_\exists(\mathbb{F}_{q})$ is decidable.
\end{proof}

For the sake of interest, we present a more direct proof of this special case. 
However, note that this `second proof' uses the decidability of $\mathbb{F}_q$,
while the `first proof' used only the decidability of the {\em existential} theory of $\mathbb{F}_q$.

\begin{proof}[Second proof]
As an equicharacteristic tame field (\autoref{lem:racl.tame}) with decidable residue field and value group, 
$(\mathbb{F}_{q}((t))^{\mathbb{Q}},v_{t})$ is decidable, by \cite[Theorem 7.7(a)]{Kuhlmann??}. 
Since $(\mathbb{F}_{q}((t))^{\mathbb{Q}},v_{t})$ is the directed union of structures 
isomorphic to $(\mathbb{F}_{q}((t)),v_{t})$ (\autoref{prp:finite.extensions.of.power.series}), 
in fact \linebreak $(\mathbb{F}_{q}((t)),v_{t})$ and $(\mathbb{F}_{q}((t))^{\mathbb{Q}},v_{t})$ have the same $\exists$-$\mathcal{L}_{\rm vf}$-theory. 
Thus, to decide the existential $\mathcal{L}_{\rm vf}$-theory of $(\mathbb{F}_{q}((t)),v_{t})$, it suffices to apply the decision procedure for the $\mathcal{L}_{\rm vf}$-theory of $(\mathbb{F}_{q}((t))^{\mathbb{Q}},v_{t})$.
\end{proof}

\begin{remark}
Since \autoref{cor:Fq((t))} shows decidability of the existential theory of $\mathbb{F}_q((t))$
in the language of {\em valued} fields $\mathcal{L}_{\rm vf}$,
in which the valuation ring is definable by a quantifier-free formula,
we also get decidability of the existential theory of the {\em ring} $\mathbb{F}_q[[t]]$.
It might however be interesting to point out that it was proven only recently that
already decidability of the existential theory of $\mathbb{F}_q((t))$ in the language of {\em rings}
would imply decidability of the existential theory of the ring $\mathbb{F}_q[[t]]$,
see \cite[Corollary 3.4]{Anscombe-Koenigsmann14}.
\end{remark}

\begin{remark}
The $\exists$-$\mathcal{L}_{\rm vf}(t)$-theory of $(\mathbb{F}_{q}((t)),v_{t})$ is equivalent to the $\forall_{1}^K\exists$-$\mathcal{L}_{\rm vf}$-theory of $(\mathbb{F}_{q}((t)),v_{t})$. This `equivalence' is meant in the sense that there is a truth-preserving effective translation between  $\exists$-$\mathcal{L}_{\rm vf}(t)$-sentences and
$\forall\exists$-$\mathcal{L}_{\rm vf}$-sentences which have only one universal quantifier ranging over the valued field sort
(and arbitrary existential quantifiers).
In this argument we make repeated use of the fact that, for all $a\in\mathbb{F}_{q}((t))$ with $v_t(a)>0$ and $a\neq0$, there is an 
$\mathcal{L}_{\rm vf}$-embedding $\mathbb{F}_{q}((t))\longrightarrow\mathbb{F}_{q}((t))$ which sends $t\longmapsto a$.

Let $\phi(t)$ be an existential $\mathcal{L}_{\rm vf}(t)$-sentence. We claim that $\phi(t)$ is equivalent to the $\forall_{1}^K\exists$-$\mathcal{L}_{\rm vf}$-sentence 
$$
 \forall u\;((v(u)>0\wedge u\neq0)\longrightarrow\phi(u)).
$$ 
This follows from the fact about embeddings, stated above.

On the other hand, let $\psi(x)$ be an $\exists$-$\mathcal{L}_{\rm vf}$-formula in one free variable $x$ in the valued field sort and consider the $\exists$-$\mathcal{L}_{\rm vf}(t)$-sentence $\chi$ which is defined to be
$$
 \exists y\exists z_{0}...\exists z_{q-1}\;(yt=1\wedge\psi(y)\wedge\bigwedge_{j}z_{j}^{q}=z_{j}\wedge\bigwedge_{i\neq j}z_{i}\neq z_{j}\wedge\bigwedge_{j}\psi(z_{j}+t)\wedge\bigwedge_{j}\psi(z_{j})).
$$
Written more informally, the sentence $\chi$ expresses that
$$
 \psi(t^{-1})\wedge\bigwedge_{z\in\mathbb{F}_{q}}(\psi(z+t)\wedge\psi(z)).
$$
We claim that $\forall x\;\psi(x)$ and $\chi$ are equivalent. First suppose that $\mathbb{F}_{q}((t))\models\forall x\;\psi(x)$. By choosing $(z_{j})$ to be an enumeration of $\mathbb{F}_{q}$, we immediately have that $\mathbb{F}_{q}((t))\models\chi$.

In the other direction, suppose that $\mathbb{F}_{q}((t))\models\chi$ and let $a\in\mathbb{F}_{q}((t))$. If $v_t(a)<0$ then consider the embedding which sends $t\longmapsto a^{-1}$. Since $\psi(t^{-1})$ holds, then $\psi(a)$ holds. 
On the other hand suppose that $v_t(a)\geq0$. If $a\in\mathbb{F}_{q}$ then $\chi$ already entails that $\psi(a)$. Now suppose that $a\notin\mathbb{F}_{q}$ and let $z$ be the residue of $a$. Consider the embedding which sends $t\longmapsto a-z$ (note that $a-z\neq0$). Since $\psi(z+t)$ holds, then $\psi(a)$ holds. This completes the proof that $\mathbb{F}_{q}((t))\models\forall x\;\psi(x)$.
\end{remark}

\section*{Acknowledgements}

\noindent
The authors would like to thank 
Immanuel Halupczok, Ehud Hrushovski, Jochen Koenigsmann, Dugald Macpherson and Alexander Prestel for helpful discussions and encouragement.

%%% BIBLIOGRAPHY %%%
\bibliographystyle{plain}
%\bibliography{arno.bibliography}

\end{document}